\documentclass[12pt,a4paper,reqno]{amsart}
\usepackage[mathcal]{eucal}
\usepackage{amssymb}
\usepackage[nice]{nicefrac}
\usepackage{hyperref}
\usepackage[normalem]{ulem}
\usepackage{color}

\theoremstyle{plain}
\newtheorem{theorem}{Theorem}[section]
\newtheorem{lemma}[theorem]{Lemma}
\newtheorem{proposition}[theorem]{Proposition}
\newtheorem{corollary}[theorem]{Corollary}

\theoremstyle{definition}

\newtheorem{remark}[theorem]{Remark}

\numberwithin{equation}{section}

\DeclareMathOperator{\hdim}{hdim}
\DeclareMathOperator{\hspec}{hspec}

\newcommand{\N}{{\ensuremath \mathbb{N}}}

\begin{document}

\title[Normal Hausdorff spectra]{Normal Hausdorff spectra of pro-$2$ groups}

\author[A. Thillaisundaram]{Anitha Thillaisundaram} 
\address{Anitha Thillaisundaram: School of Mathematics and Physics,
  University of Lincoln, Lincoln LN6 7TS, England}
\email{anitha.t@cantab.net}

\date{\today}





\begin{abstract} 
Klopsch and the author have constructed  a finitely generated pro-$p$
    group $G$, for $p$ an odd prime,  with infinite normal Hausdorff spectrum
  \[
  \hspec_\trianglelefteq^\mathcal{P}(G) = \{ \hdim_G^\mathcal{P}(H)
  \mid H \trianglelefteq_\mathrm{c} G \};
  \]
  here
  $\hdim_G^\mathcal{P} \colon \{ X \mid X \subseteq G \} \to [0,1]$
  denotes the Hausdorff dimension function associated to the $p$-power
  series $\mathcal{P} \colon G^{p^i}$, $i \in \mathbb{N}_0$. They show that
  $\hspec_\trianglelefteq^\mathcal{P}(G) = [0,\nicefrac{1}{3}] \cup
  \{1\}$
  contains an infinite interval, which answers a question of Shalev.

 They indicate in their paper how their results extend to the case $p=2$. In this note, we provide all the details for the even case.
 
\end{abstract}

\maketitle


\section{Introduction}

This paper supplements the corresponding paper \cite{KT} by Klopsch and the author for the odd prime case, and we refer the reader to \cite{KT} for notation and terminology used. In this paper, we are interested in a group constructed as follows.
The pro\nobreakdash-$2$ wreath product
$W = C_2 \mathrel{\hat{\wr}} \mathbb{Z}_2$ is the inverse limit
$\varprojlim_{n \in \N} C_2 \mathrel{\hat{\wr}}  C_{2^n}$ of the finite standard
wreath products of cyclic groups with respect to the natural
projections; clearly, $W$ is $2$-generated as a topological group.  Let
$F$ be the free pro-$2$ group on two generators and let
$R \trianglelefteq_\mathrm{c} F$ be the kernel of a presentation
$\pi \colon F \to W$.  We are interested in the pro-$2$ group
\[
G = F/N, \qquad \text{where} \quad N = [R,F] R^2
\trianglelefteq_\mathrm{c} F.
\]

\begin{theorem} \label{thm:main-thm} The normal Hausdorff
  spectra of the pro-$2$ group $G$ constructed above, with respect to
  the standard  filtrations series $\mathcal{P}$, $\mathcal{D}$, $\mathcal{F}$ and $\mathcal{L}$ respectively, 
  satisfy:  
  \begin{align*}
  \textup{hspec}^{\mathcal{P}}_{\trianglelefteq}(G) &= \textup{hspec}^{\mathcal{D}}_{\trianglelefteq}(G) = 
  \textup{hspec}^{\mathcal{F}}_{\trianglelefteq}(G)  = [0,1/3] \cup \{1\} \\
 \textup{hspec}^{\mathcal{L}}_{\trianglelefteq}(G) &= [0,1/5] \cup \{3/5\}\cup \{1\}
  \end{align*}
In particular, they contain an infinite real interval.
\end{theorem}

This also solves a problem posed by Shalev~\cite[Problem~16]{Sh00}.
As is done in \cite{KT}, we further compute the entire Hausdorff spectra of $G$ with respect to
the four standard filtration series.

\begin{theorem} \label{thm:entire-spectrum}  The Hausdorff
  spectra of the pro-$2$ group $G$ constructed above, with respect to
  the standard filtration series, satisfy:
  \begin{align*}
    \hspec^{\mathcal{P}}(G) & = \hspec^{\mathcal{D}}(G) =
                              \hspec^{\mathcal{F}}(G) = [0,1],\\
    \hspec^{\mathcal{L}}(G) & =  [0,\nicefrac{4}{5}) \cup
                              \{\nicefrac{3}{5} + \nicefrac{m}{5\cdot 2^{n-1}}\mid
                              m, n \in\mathbb{N}_0 \text{ with }
                              2^{n-1} < m\le 2^n\}.
  \end{align*}
\end{theorem}

This paper is organised as follows. Section 2 gives an explicit 
presentation of the pro-$2$ group $G$ and describes its finite quotients $G_k$ for $k\in \mathbb{N}$.  In
Section~\ref{sec:p-power-series} we compute the normal Hausdorff
spectrum of $G$ with respect to the $2$-power series~$\mathcal{P}$,
and in Section~\ref{sec:other-series} we compute the normal Hausdorff
spectra of $G$ with respect to the other three standard filtration
series~$\mathcal{D}, \mathcal{F}, \mathcal{L}$.  Finally, in Section~\ref{sec:entire-spectrum} we compute the entire Hausdorff spectra of $G$.


\section{An explicit presentation for the pro-$2$ group~$G$ and \\ a
  description of its finite quotients $G_k$ for $k \in
  \mathbb{N}$} \label{sec:presentation}

 As indicated in the paragraph
before Theorem~\ref{thm:main-thm}, we consider the pro-$2$ group
$G = F/N$, where 
\begin{itemize}
\item $F = \langle x,y \rangle$ is a free pro-$2$ group and
\item $N = [R,F] R^2 \trianglelefteq_\mathrm{c} F$ for the kernel
  $R \trianglelefteq_\mathrm{c} F$ of the presentation
  $\pi \colon F \to W$ sending $x,y$ to the generators of the same
  name in $W$.
\end{itemize}
As in \cite{KT}, we now obtain explicit presentations for the pro-$2$ groups $W$
and~$G$.

We write $y_i = y^{x^i}$ for $i \in \mathbb{Z}$.
Setting
\begin{equation}\label{equ:R_k}
  R_k =\langle \{ x^{2^k}, y^2 \} \cup \{ [y_0,y_i] \mid 1\le i \le
  2^{k-1} \} \rangle ^F \trianglelefteq_\mathrm{o} F
\end{equation}
for $k\in \mathbb{N}$, we obtain a descending chain of open normal
subgroups
\begin{equation}\label{equ:R}
  F \supseteq R_1\supseteq R_2 \supseteq \ldots
\end{equation}
with quotient groups
$F/R_k \cong W_k \cong C_2 \mathrel{\wr} C_{2^k}$.  Writing
\begin{equation*}
  R=\bigcap\nolimits_{k \in \mathbb{N}} R_k = \langle  \{ y^2 \} \cup
  \{ [y_0,y_i] \mid i\in \mathbb{N} \} \rangle ^F \trianglelefteq_\mathrm{c} F, 
\end{equation*}
we obtain $F/R \cong W \cong C_2 \mathrel{\hat{\wr}} \mathbb{Z}_2$.  

Setting $N_k = [R_k,F]R_k^{\, 2}$ for $k \in \mathbb{N}$, we observe
that
\begin{multline*}
  N_k = \langle \{ x^{2^{k+1}} , y^{4}, [x^{2^k},y], [y^2,x] \} \cup
  \{ [y_0,y_i]^2 \mid 1\le i \le 2^{k-1} \}  \\
  \cup \{ [y_0,y_i,x] \mid 1\le i \le 2^{k-1}\} \cup \{
  [y_0,y_i,y] \mid 1\le i \le 2^{k-1}\} \rangle ^F
  \trianglelefteq_\mathrm{o} F,
\end{multline*}
and as in \eqref{equ:R} we obtain a descending chain
$F \supseteq N_1 \supseteq N_2 \supseteq \ldots$ of open normal
subgroups.  As in \cite{KT}, we conclude that
\begin{equation*}
  \bigcap\nolimits_{k \in \mathbb{N}} N_k = [R,F]R^2 = N.
\end{equation*}

Consequently, $G = F/N \cong \varprojlim G_k$, where
\begin{equation}\label{eq:pres-for-G-k}
  \begin{split}
    G_k = F/N_k
    & \cong  \langle x,y \mid \uline{\phantom{[}x^{2^{k+1}}},\;
    y^{4},\; \uline{~[x^{2^k},y]},\; [y^2,x];\\
    & \qquad \quad [y_0,y_i]^2,\; [y_0,y_i,x],\; [y_0,y_i,y] \quad
    \text{for $1\le i\le 2^{k-1}$} \rangle
  \end{split}
\end{equation}
for $k \in \mathbb{N}$, and 
\begin{equation}\label{equ:pres-for-G}
  G \cong \langle x,y \mid y^{4},\; [y^2,x];\; [y_0,y_i]^2,\;
  [y_0,y_i,x],\; [y_0,y_i,y] \quad \text{for $i \in \mathbb{N}$} \rangle
\end{equation}
is a presentation of $G$ as a pro-$2$ group.  To facilitate later use,
we have underlined the two relations in \eqref{eq:pres-for-G-k} that
do not yet occur in~\eqref{equ:pres-for-G}.

\medskip

We note  that all results in \cite[\S 2 and \S 4]{KT} and \cite[Lem.~3.1]{KT} hold for $p=2$ with corresponding proofs. 


\section{The normal Hausdorff spectrum of $G$ with respect to the $2$-power series}\label{sec:p-power-series}

For convenience we recall two standard commutator
collection formulae.

\begin{proposition}\label{pro:standard-commutator-id}
  For a prime $p$, let $G = \langle a,b \rangle$ be a finite $p$-group, and let
  $r \in \N$.  For $u,v \in G$ let $K(u,v)$ denote the normal closure
  in $G$ of \textup{(i)} all commutators in $\{u,v\}$ of weight at
  least $p^r$ that have weight at least $2$ in $v$, together with
  \textup{(ii)} the $p^{r-s+1}$th powers of all commutators in
  $\{u,v\}$ of weight less than $p^s$ and of weight at least $2$
  in~$v$ for $1 \le s \leq r$.  Then
  \begin{align}
    (ab)^{p^r} & \equiv_{K(a,b)} a^{p^r} \, b^{p^r} \, [b, a]^{\binom{p^r}{2}} \,
             [b, a, a]^{\binom{p^r}{3}} \, \cdots \,
             [b, a, \overset{p^r-2} \ldots, a]^{\binom{p^r}{p^r-1}} \,
             [b, a, \overset{p^r-1} \ldots,
             a], \label{equ:commutator-formula-1} \\ 
    [a^{p^r}, b] & \equiv_{K(a,[a,b])} [a, b]^{p^r} \, [a,b,a]^{\binom{p^r}{2}} \,
               \cdots \, [a, b, a, \overset{p^r-2} \ldots, a]^{\binom{p^r}{p^r-1}}
               \, [a, b, a, \overset{p^r-1} \ldots, a]. \label{equ:commutator-formula-2} 
  \end{align}
\end{proposition}

 This result is recorded (in a slighter stronger form) in
\cite[Prop.~1.1.32]{LeMc02}; we remark that
\eqref{equ:commutator-formula-2} follows  from
\eqref{equ:commutator-formula-1}, due to the identity
$[a^{p^r},b] = a^{-p^r} (a [a,b])^{p^r}$.

\medskip

Now we continue to use the notation set up in
Section~\ref{sec:presentation} and establish that
$\xi = \hdim^{\mathcal{P}}_G(Z) = \nicefrac{1}{3}$ and
$\eta = \hdim^{\mathcal{P}}_G(H) = 1$, with respect to the $2$-power
series~$\mathcal{P}$.  In view of
\cite[Cor. 4.2]{KT} this proves
Theorem~\ref{thm:main-thm} for the $2$-power series.  Indeed,
$\hdim^{\mathcal{P}}_G(H) = 1$ is already a consequence of
\cite[Prop.~4.2]{KlThZuXX}.  It remains to show that
\begin{equation} \label{equ:xi-1/3-P} \hdim^{\mathcal{P}}_G(Z) =
  \varliminf_{i\to \infty} \frac{\log_2 \lvert ZG^{2^i} : G^{2^i}
    \rvert}{\log_2 \lvert G : G^{2^i} \rvert} = \nicefrac{1}{3}.
\end{equation}

As in \cite{KT} we work with the finite quotients $G_k$,
$k \in \mathbb{N}$, introduced in Section~\ref{sec:presentation}.  Let
$k \in \mathbb{N}$.  From \eqref{eq:pres-for-G-k} and
\eqref{equ:pres-for-G} we observe that
\[ 
\lvert G : G^{2^k} \rvert = \lvert G_k : G_k^{\, 2^k} \rvert.
\]
First we compute the order of~$G_k$, using the notation from
Section~\ref{sec:presentation}. 

\begin{lemma}\label{lem:order-Gk}
  The logarithmic order of $G_k$ is
  \[
  \log_2 \lvert G_k \rvert = 2^k+2^{k-1}+k+2.
  \]
  In particular,  
 \begin{equation*}
   Z_k = R_k/N_k = \langle \{ x^{2^k}, y^2 \} \cup \{ [y_0,y_i] \mid 1 \le i \le
    2^{k-1} \} \rangle N_k / N_k \cong C_2 \times
   \overset{2^{k-1}+2}{\ldots} \times C_2. 
 \end{equation*}
\end{lemma}

\begin{proof}
  This is essentially the same proof as for \cite[Lem. 5.1]{KT}.
\end{proof}

Our next aim is to prove the following structural result.

\begin{proposition} \label{pro:index-Gk-Gk-hoch-p} In the set-up from
  Section~\ref{sec:presentation}, for $k\ge 2$, the subgroup
  $G_k^{\, 2^k} \le G_k$ is elementary abelian;
  it is generated independently by $x^{2^k}$,
  $w = y_{2^k-1} \cdots y_1y_0$ and
  $[w,x]=[y_0,y_{2^{k-1}}]$.

  Consequently
  \[
  G_k^{\, 2^k} \cong C_2 \times C_2 \times C_2, \qquad \log_2 \vert
  G_k : G_k^{\, 2^k} \rvert = \log_2 \lvert G_k \rvert - 3
  \]
  and
  \begin{equation*}
    \begin{split}
      G_k / G_k^{\, 2^k} & \cong \langle x,y \mid x^{2^{k}},\;
      y^{4},\; 
      [y^2,x],\; w(x,y),\; [y_0,y_{2^{k-1}}];\\
      & \qquad \quad [y_0,y_i]^2,\; [y_0,y_i,x],\; [y_0,y_i,y] \quad
      \textup{for } 1\le i\le 2^{k-1} \rangle.
    \end{split}
  \end{equation*}
\end{proposition}

The proof requires a series of lemmata.

\begin{lemma} \label{lem:w-elements} The elements
  \[
    w = y_{2^k-1}\cdots y_1y_0 \qquad \text{and} \qquad
    [w,x]=[w,y]=[y_0,y_{2^{k-1}}]
  \]
  are of order $2$ in $G_k$ and lie in $G_k^{\, 2^k}$.
    In particular the subgroup $\langle x^{2^k},w,[w,x]\rangle$ is isomorphic to $C_2\times C_2\times C_2$ and lies in
     $G_k^{\, 2^k}$.
\end{lemma}

\begin{proof}
  Recall $H_k = \langle y_0, y_1, \ldots, y_{2^k-1} \rangle \le G_k$ and
  observe that $[H_k,H_k]$ is a central subgroup of exponent $2$
  in~$G_k$.  Furthermore, $[y^2,x] =1$ implies
  $y_{2^k-1}^{\, 2} = \ldots = y_0^{\, 2}$ in~$G_k$.  Also we observe that, for
  $1 \le i \le 2^k-1$, the relation $[y_0,y_i,x]=1$ implies
 \begin{equation}\label{eq:note}
   [y_0,y_{2^k-i}]^{-1} = [y_{2^k-i},y_0] = [y_0,y_i]^{x^{-i}} =
   [y_0,y_i] \qquad \text{in $G_k$.} 
 \end{equation}
  
  Thus Proposition \ref{pro:standard-commutator-id}(i) and (\ref{eq:note}) yield
  \begin{align*}
  w^2 &= ( y_{2^k-1}\cdots y_1y_0 )^2\\
  &= (y_{2^k-1}\cdots y_1)^2y_0^2[y_0,y_{2^{k}-1}]\ldots [y_0,y_{2^{k-1}+1}][y_0,y_{2^{k-1}}]
  [y_0,y_{2^{k-1}-1}]\ldots[y_0,y_1]\\
    &\,\,\vdots \\
  &= y_{2^k-1}^{\, 2} \cdots y_1^{\, 2} y_0^{\, 2}\\ 
  &= y^{2^{k+1}} \\
  &= 1.
  \end{align*}
  As $w \ne 1$ we deduce that $w$ has order~$2$.   Clearly, $w = x^{-2^k} (xy)^{2^k}$, so both $w$ and
  $[w,x]$ lie in $G_k^{\, 2^k}$.  
  
  Next, 
 \begin{align*}
   [w,x] & = (y_{2^k-1} \cdots
           y_1y_0)^{-1} (y_{2^k-1} \cdots y_1 y_0)^x \\
         &  = y_0^{\, -1} y_1^{\, -1} \cdots y_{2^k-2}^{\, -1} \cdot
           y_{2^k-1}^{\, -1} y_0 y_{2^k-1} \cdot y_{2^k-2} \cdots
           y_2 y_1 \\
         & = y_0^{\, -1} y_1^{\, -1} \cdots y_{2^k-2}^{-1} \cdot y_0 [y_0,
           y_{2^k-1}] \cdot y_{2^k-2} \cdots y_2y_1 \\
         & = y_0^{-1} y_1^{-1} \cdots y_{2^k-3}^{-1} \cdot
           y_{2^k-2}^{\, -1} y_0 y_{2^k-2} \cdot
           y_{2^k-3} \cdots y_2 y_1 \cdot [y_0,y_{2^k-1}]\\
         & \quad \vdots \\
         & = [y_0,y_1] [y_0,y_2] \cdots [y_0,y_{2^k-2}] [y_0,y_{2^k-1}]\\
         & =[y_0,y_{2^{k-1}}] && \text{by \eqref{eq:note}} \\
         & =[w,y],         
  \end{align*}
  as required.
\end{proof}

\begin{lemma} \label{lem:gamma-2-exp-4} The group
  $\gamma_2(G_k) \le G_k$ has exponent~$4$.
\end{lemma}

\begin{proof}
  Recall that
  $H_k = \langle y_0, y_1, \ldots, y_{2^k-1} \rangle Z_k \le G_k$
  satisfies: $[H_k,H_k]$ is a central subgroup of exponent $2$
  in~$G_k$.  Hence,
  \eqref{equ:commutator-formula-1} shows that it
  suffices to prove that $[y,x]$ has order~$4$.  But
  $[y,x] = y_0^{\, -1} y_1$; thus 
  $y_0^{\, 2} = x^{-1} y_0^{\, 2} x = y_1^{\, 2}$ implies
  $[y,x]^2 = y_0^{\, -2} y_1^{\, 2} [y_1,y_0^{-1}]=[y_1,y_0]$, and therefore  
  $[y,x]^4 = [y_1,y_0]^2=1$.
\end{proof}

\begin{lemma}\label{lem:squared-commutators}
In the group $G_k$, the following holds:
\[
[y,x,\overset{i}\ldots,x]^2\in [H_k,H_k]\cap \gamma_{2i+1}(G_k),\quad \text{ for } i\ge 1.
\]
Furthermore, $[y,x,\overset{i}\ldots,x]^2\equiv [y,x,\overset{2i-1}\ldots,x,y]$ modulo $\gamma_{2i+2}(G_k)$.
\end{lemma}

\begin{proof}
For $i=1$, we only need to show that $[y,x]^2\in \gamma_3(G_k)$, as from the previous proof, we already have $[y,x]^2=[y_1,y_0]\in [H_k,H_k]$. Now
\begin{align*}
[y,x]^2&=[y_1,y_0] =[x^{-1}yx,y]\equiv [x^{-1},y][x,y]\quad \text{mod }\gamma_3(G_k)\\
&\equiv [x,y]^{2^{k+1}-1}[x,y]\quad \text{mod }\gamma_3(G_k)\\
&= 1 \quad \text{mod }\gamma_3(G_k),
\end{align*}
by Lemma~\ref{lem:gamma-2-exp-4}.

Next, for $i\ge 2$, we have
\[
[[y,x,\overset{i-1}\ldots, x],x]^2=([y,x,\overset{i-1}\ldots, x]^{-1})^2 ([y,x,\overset{i-1}\ldots, x]^x)^2 [[y,x,\overset{i-1}\ldots, x]^x, [y,x,\overset{i-1}\ldots, x]^{-1}].
\]
By induction $[y,x,\overset{i-1}\ldots, x]^2\in [H_k,H_k]\le Z(G_k)$, hence
\begin{align*}
[[y,x,\overset{i-1}\ldots, x],x]^2&=[[y,x,\overset{i-1}\ldots, x]^x, [y,x,\overset{i-1}\ldots, x]^{-1}]\\
&=[[y,x,\overset{i-1}\ldots, x][y,x,\overset{i}\ldots, x], [y,x,\overset{i-1}\ldots, x]^{-1}]\\
&=[[y,x,\overset{i}\ldots, x], [y,x,\overset{i-1}\ldots, x]^{-1}]\in  [H_k,H_k]\cap \gamma_{2i+1}(G_k),
\end{align*}
as required.

For the final statement, for ease of notation we set $c_1=y$ and, for $i \ge 2$,
  \[
  c_i = [y,x,\overset{i-1}{\ldots},x] \qquad \text{and} \qquad z_i =
  [c_{i-1},y] = [y,x,\overset{i-2}{\ldots},x,y].
  \]
  We observe that
  $z_i^{\, 2} = 1$ for $i \ge 2$; furthermore, the
  elements $z_i \in [H_k,H_k] \subseteq Z_k$ are central in~$G_k$.  We
  claim that
  \begin{equation} \label{equ:comm-ci-cj}
    [c_i,c_j] \equiv z_{i+j} \;\mathrm{mod}\;
      \gamma_{i+j+1}(G_k) \qquad \text{for $i > j \ge 1$.}  
  \end{equation}
  Indeed $[c_i,c_1] = [c_i,y] = z_{i+1}$ and, modulo
  $\gamma_{i+j+1}(G_k)$, the Hall--Witt identity gives
  \[
  1 \equiv [c_i,c_{j-1},x] [c_{j-1},x,c_i] [x,c_i,c_{j-1}] \equiv [c_j,c_i]
  [c_{i+1},c_{j-1}],
  \]
  hence $[c_i,c_j] \equiv [c_{i+1},c_{j-1}]$ from which the result
  follows by induction.
\end{proof}

\begin{lemma} \label{lem:contained} The group $G_k$ satisfies $G_k^{\, 2} \subseteq \langle x^{2}, y^2\rangle\gamma_{2}(G_k)$ and
\begin{equation*}
 G_k^{\, 2^j} \subseteq \langle x^{2^j},[y,x,\overset{2^{j}-3}\ldots,x,y]\rangle\gamma_{2^{j}}(G_k) \subseteq \langle x^{2^j}\rangle\gamma_{2^{j}-1}(G_k)\quad \text{ for }j\ge 2.
  \end{equation*}
In particular,
  \begin{equation*}
        G_k^{\, 2^k} \subseteq \langle x^{2^k},w,[w,x]\rangle.
  \end{equation*}
\end{lemma}

\begin{proof}
 For $j\ge 2$, we note that $(xy)^{2^j}=x^{2^j}[y,x,x^2,\ldots,x^{2^{j-2}}]^2[y,x,x^2,\ldots,x^{2^{j-1}}]$ by repeated application of the rule $(ab)^2=a^2b^2[b,a]$, for $a,b\in G_k$, plus using the fact that $\gamma_2(G_k)^2\le [H_k,H_k]\le Z(G_k)$. 
  
 Next, a 
 general element of $G_k$ is of the form $x^mh$ for some $m\in \mathbb{N}_0$ and $h\in H_k$. 
 As before, we have 
 $$(x^mh)^{2^j}=x^{2^jm}[h,x^m,(x^m)^2,\ldots,(x^m)^{2^{j-2}}]^2[h,x^m,(x^m)^2,\ldots,(x^m)^{2^{j-1}}].$$ Since $[y,x,x^2,\ldots,x^{2^{j-2}}]^2= [y,x,\overset{2^{j-1}-1}\ldots,x]^2$,  we deduce from Lemma~\ref{lem:squared-commutators} that $(x^mh)^{2^j}\in \langle x^{2^j},
 [y,x,\overset{2^{j}-3}\ldots,x,y]\rangle \gamma_{2^j}(G_k)$. Hence the result.
 
 The final statement follows from $w=[y,x,x^2,\ldots,x^{2^{k-2}}]^2[y,x,x^2,\ldots,x^{2^{k-1}}]$. 
\end{proof}

We state the following result for completeness.
\begin{lemma} \label{lem:gamma-m-modulo-gamma-m+1}
  The group $G_k$ has nilpotency class~$2^k+1$, and
  $\gamma_m(G_k) / \gamma_{m+1}(G_k)$ is elementary abelian of rank at
  most~$2$ for $2 \le m \le 2^k+1$.
\end{lemma}

\begin{proof}
   Let $2 \le m \le 2^k+1$.  Since $G_k$ is a central extension of $Z_k$
  by $W_k$, we deduce from \cite[Prop.~2.6]{KT} that
  \[
  \gamma_m(G_k) = \langle [y,x,\overset{m-1}{\ldots},x],
  [y,x,\overset{m-2}{\ldots},x,y] \rangle \; \gamma_{m+1}(G_k),
  \]
  and Lemma~\ref{lem:squared-commutators} shows that
  $\gamma_m(G_k) / \gamma_{m+1}(G_k)$ is elementary abelian of rank at
  most~$2$.  Again by \cite[Prop.~2.6]{KT}, the nilpotency
  class of $G_k$ is at least~$2^k$.  Moreover,
  $\gamma_{2^k+1}(G_k) Z_k = \langle [w,x] \rangle Z_k$.  We conclude
  that $G_k$ has nilpotency class precisely~$2^k+1$.
\end{proof}

\begin{proof}[Proof of Proposition~\ref{pro:index-Gk-Gk-hoch-p}]
  Apply
  Lemmata~\ref{lem:w-elements}
  and~\ref{lem:contained}.
\end{proof}

From Lemma~\ref{lem:order-Gk} and
Proposition~\ref{pro:index-Gk-Gk-hoch-p} we deduce that
\[
\log_2 \lvert G : G^{2^k} \rvert = \log_2 \lvert G_k : G_k^{\, 2^k}
\rvert = 2^k+2^{k-1}+k-1. 
\]
On the other hand, we observe from \cite[Prop.~2.6]{KT}
that
\[
\log_2 \lvert G : Z G^{2^k} \rvert = \log_2 \lvert W_k : W_k^{\, 2^k}
\rvert =  2^k+k-1,
\]
hence
\[
\log_2 \lvert ZG^{2^k} : G^{2^k} \rvert = 2^k+2^{k-1}+k-1 -
(2^k+k-1) = 2^{k-1}.
\]

Thus \eqref{equ:xi-1/3-P} follows from
\begin{equation} \label{equ:xi-for-P}
 \varliminf_{i\to \infty} \frac{\log_2 \lvert ZG^{2^i} : G^{2^i}
    \rvert}{\log_2 \lvert G : G^{2^i} \rvert} =
\lim_{i \to \infty} \frac{2^{i-1}}{2^i+2^{i-1}+i-1} = \nicefrac{1}{3}.
\end{equation}

\medskip

\begin{remark}
 As noted in \cite{KT}, one sometimes encounters a variant of the $2$-power series in the literature, the \emph{iterated} $2$-power series of $G$, which is recursively given by
  \[
  \mathcal{I} \colon I_0(G) = G, \quad \text{and} \quad I_j(G) =
  I_{j-1}(G)^2 \quad \text{for $j \geq 1$.}
  \]
  By a small modification of the proof of Lemma~\ref{lem:contained} we
  obtain inductively 
    \begin{equation*}
    I_j(G_k) \subseteq \big(\langle x^{2^{j-1}} \rangle
       \gamma_{2^{j-1}}(G_k) \big)^2 \subseteq  \langle x^{2^j}, [y,x,\overset{2^{j}-3}\ldots,x,y] \rangle
       \gamma_{2^j}(G_k) \quad \text{for $j \ge 2$,}
  \end{equation*}
  and furthermore $ I_k(G_k) \subseteq  \langle x^{2^k}, w, [w,x]\rangle$.
   With Proposition~\ref{pro:index-Gk-Gk-hoch-p}  this yields
  $G_k^{\, 2^k} \subseteq I_k(G_k) \subseteq   
 \langle x^{2^k}, w, [w,x]\rangle = G_k^{\, 2^k}$.  We conclude that  
    the $2$-power series $\mathcal{P}$ and the iterated $2$-power series $\mathcal{I}$ of~$G$ coincide.

  Further we note another natural filtration series
  $\mathcal{N} \colon N_i$, $i\in \mathbb{N}_0$, of $G$, consisting of
  the open normal subgroups defined in Section~\ref{sec:presentation},
  where we set $N_0=G$.  As $N_i \le G^{2^i}$ with
    $\log_2 \lvert G^{2^i} : N_i \rvert \le 4$
    for all
  $i \in \mathbb{N}_0$, we see that the
  filtration series $\mathcal{P}$ and $\mathcal{N}$ induce the same
  Hausdorff dimension function on~$G$.
\end{remark}


\section{The normal Hausdorff spectra of $G$ with respect to the
  lower $2$-series, the dimension subgroup series and the Frattini series}\label{sec:other-series}

We continue to use the notation set up in
Section~\ref{sec:presentation} and work with the finite quotients
$G_k$, $k \in \mathbb{N}$, of the pro-$2$ group~$G$.  In order to complete the proof of Theorem~\ref{thm:main-thm}, we first compute
 the lower central series, the lower $2$-series, the dimension
subgroup series and the Frattini series of~$G_k$.

\begin{proposition} \label{pro:lower-central-Gk} The group $G_k$ is
  nilpotent of class $2^k+1$; its lower central series satisfies
  \[
  G_k = \gamma_1(G_k) = \langle x,y \rangle \; \gamma_2(G_k) \quad
  \text{with} \quad G_k/\gamma_2(G_k)  \cong C_{2^{k+1}} \times
  C_{4}
 \]
 and, for $1 \le i \le 2^{k-1}$,
 \begin{align*}
   \gamma_{2i}(G_k) & = \quad \langle [y,x,\overset{2i-1}{\ldots},x] \rangle
                      \; \gamma_{2i+1}(G_k), \\
   \gamma_{2i+1}(G_k) & = \left\{\begin{array}{cc}
                                \! \langle [y,x,\overset{2i}{\ldots},x],
                        [y,x,\overset{2i-1}{\ldots},x,y] \rangle \;
                        \gamma_{2i+2}(G_k) & \text{ for }i\ne 2^{k-1}\\
                        \langle [y,x,\overset{2i}{\ldots},x] \rangle \;
                        \gamma_{2i+2}(G_k) \qquad \qquad \qquad \,\,& \text{ for }i= 2^{k-1}
                                \end{array}\right.
 \end{align*} 
 with
 \[
 \gamma_{2i}(G_k)/\gamma_{2i+1}(G_k) \cong C_2 \quad \text{and} \quad
 \gamma_{2i+1}(G_k)/\gamma_{2i+2}(G_k) \cong \left\{\begin{array}{cc}
						    C_2 \times C_2 & \text{ for }i\ne 2^{k-1}\\
						    C_2 &\text{ for }i=2^{k-1}.                                                     
                                                    \end{array}\right.
 \]
 \end{proposition}

\begin{proof}
This is the same as for \cite[Lem. 6.1]{KT} up to the equation
  \begin{equation} \label{equ:recursion}
  b_{j,m} \equiv b_{j,m} b_{j,m+1}^{\, 2} b_{j,m+2} = b_{j+2,m} \quad
  \text{modulo $\gamma_{m+1}(G_k)$.}
  \end{equation}
 As in \cite{KT}, we now suppose that $m\le 2^{k}$ is even.
  From~\eqref{equ:recursion} we obtain inductively
  $[y,x,\overset{m-2}{\ldots},x,y] = b_{0,m} \equiv b_{j_0,m}$ modulo
  $\gamma_{m+1}(G_k)$ for
  \[
  j_0 =
  \begin{cases}
   2^{k-1} - \frac{m}{2} & \text{if $m \equiv_4
      0$,} \\
    2^{k-1} +1 - \frac{m}{2} & \text{if $m \equiv_4
      2$.}
  \end{cases}
  \]

  Consequently, it suffices to prove that
  $b_{j_0,m} \in \gamma_{m+1}(G_k)$.  We first assume
  $m \equiv_4 0$ so that $j_0 = 2^{k-1} - \frac{m}{2}$.  In this situation, we
  see, as in \cite{KT}, and the fact that the $e_i$ have exponent 2, that
  
  \begin{align*}
    b_{j_0,m} & = \prod\nolimits_{i=0}^{m/2} e_{j_0+i}^{\,
                 \binom{m-2}{i}} \prod\nolimits_{i=
                m/2+1}^{m-2} e_{2^k - (j_0+i)}^{\, 
                 \binom{m-2}{i}} \\
              & = \prod\nolimits_{i=0}^{m/2} e_{j_0+i}^{\,
                 \binom{m-2}{i}} \prod\nolimits_{i=
                m/2+1}^{m-2} e_{j_0 + (m-i)}^{\, 
                 \binom{m-2}{(m-i)-2}} \\
              & = \prod\nolimits_{i=0}^{m/2} e_{j_0+i}^{\,
                 \binom{m-2}{i}} \prod\nolimits_{i'=
                2}^{m/2-1} e_{j_0 + i'}^{\, 
                 \binom{m-2}{i'-2}} \\
              & = e_{j_0}e_{2^{k-1}}^{\,\binom{m-2}{\nicefrac{m}{2}}}\prod\nolimits_{i=2}^{m/2-1} e_{j_0+i}^{\,
                \big[ \binom{m-2}{i-2}-\binom{m-2}{i} \big]}\\
              & = e_{j_0}e_{2^{k-1}}^{\,\binom{m-2}{\nicefrac{m}{2}}}\prod\nolimits_{i=2}^{m/2-1} e_{j_0+i}^{\,       \binom{m}{i}}   \\
              & \equiv         e_{j_0} \prod\nolimits_{i=2}^{m/2-1} e_{j_0+i}^{\,       \binom{m}{i}} \qquad \text{mod }\gamma_{m+1}(G_k),                
  \end{align*}
  since $e_{2^{k-1}}=[w,x]\in \gamma_{2^k+1}(G_k)$, 
  and similarly
  \begin{align*}
    b_{j_0,m+1} & = \prod\nolimits_{i=0}^{m/2} e_{j_0+i}^{\,
                 \binom{m-1}{i}} \prod\nolimits_{i=
                m/2+1}^{m-1} e_{2^k - (j_0+i)}^{\, 
                 \binom{m-1}{i}} \\
              & = \prod\nolimits_{i=0}^{m/2} e_{j_0+i}^{\,
                \binom{m-1}{i}} \prod\nolimits_{i=
                m/2+1}^{m-1} e_{j_0 + (m-i)}^{\, 
                 \binom{m-1}{m-i-1}} \\
              & = \prod\nolimits_{i=0}^{m/2} e_{j_0+i}^{\,
                \binom{m-1}{i}} \prod\nolimits_{i'=
                1}^{m/2-1} e_{j_0 + i'}^{\, 
                 \binom{m-1}{i'-1}} \\
              & = e_{j_0}  e_{2^{k-1}}^{\,\binom{m-1}{\nicefrac{m}{2}}}\prod\nolimits_{i=1}^{m/2-1} e_{j_0+i}^{\,
                 \big[\binom{m-1}{i}+\binom{m-1}{i-1}\big]}\\
              & = e_{j_0}  e_{2^{k-1}}^{\,\binom{m-1}{\nicefrac{m}{2}}}
              \prod\nolimits_{i=2}^{m/2-1} e_{j_0+i}^{\,
                 \binom{m}{i}} \\
                 & \equiv         e_{j_0} \prod\nolimits_{i=2}^{m/2-1} e_{j_0+i}^{\,       \binom{m}{i}} \qquad \text{mod }\gamma_{m+1}(G_k).    
  \end{align*}
  Hence,  $b_{j_0,m} \equiv b_{j_0,m+1}$ modulo 
  $\gamma_{m+1}(G_k)$. Thus $b_{j_0,m} \in \gamma_{m+1}(G_k)$.
  
  Lastly, for  the case $m \equiv_4 2$ and  $j_0 = 2^{k-1}+1 - \frac{m}{2}$, we have
   \begin{align*}
    b_{j_0,m} & = \prod\nolimits_{i=0}^{m/2-1} e_{j_0+i}^{\,
                 \binom{m-2}{i}} \prod\nolimits_{i=
                m/2}^{m-2} e_{2^k - (j_0+i)}^{\, 
                 \binom{m-2}{i}} \\
              & = \prod\nolimits_{i=0}^{m/2-1} e_{j_0+i}^{\,
                 \binom{m-2}{i}} \prod\nolimits_{i=
                m/2}^{m-2} e_{j_0 + (m-2-i)}^{\, 
                 \binom{m-2}{m-2-i}} \\
              & = \prod\nolimits_{i=0}^{m/2-1} e_{j_0+i}^{\,
                 \binom{m-2}{i}} \prod\nolimits_{i'=
               0}^{m/2-2} e_{j_0 + i'}^{\, 
                 \binom{m-2}{i'}}\\
              & = e_{2^{k-1}}^{\,\binom{m-2}{\nicefrac{m}{2}-1}}\in\gamma_{m+1}(G_k).                
  \end{align*}
\end{proof}

\begin{corollary} \label{cor:gamma-central-part}
  For $2 \le m \le 2^k$ and
  $\nu(m) = \lceil 2^{k-1}-\frac{m}{2}+1 \rceil$, we have
  \[
    \gamma_m(G_k) \cap Z_k = \langle [y,x,\overset{2j-1}{\ldots},x,y]
    \mid \lfloor \nicefrac{m}{2} \rfloor \le j \le
    2^{k-1} \rangle \cong C_2^{\, \nu(m)}
  \]
  and
  $\gamma_m(G_k) \cap Z(G_k) = \langle [y,x,\overset{2^k}{\ldots},x]
  \rangle \times (\gamma_m(G_k) \cap Z_k) \cong C_2^{\, \nu(m)+1}$.
  In particular,
  \[
  [y,x,\overset{m-2}{\ldots},x,y] \in \langle [y,x,\overset{2j-1}{\ldots},x,y]
  \mid \nicefrac{m}{2} \le j \le 2^{k-1}
  \rangle \qquad \text{for $m \equiv_2 0$.}
  \]
\end{corollary}

\begin{proof}
 This follows as in \cite[Proof of Cor.~6.2]{KT}.
\end{proof}

\begin{corollary} \label{cor:lower-p-central-Gk} The lower $2$-series
  of $G_k$ has length $2^k+1$ and satisfies:
  \begin{align*}
    & G_k  = P_1(G_k)  = \langle x,y \rangle \; P_2(G_k) %
    &  \text{with} \quad %
    & G_k/P_2(G_k)  \cong C_2 \times C_2, \\
    & P_2(G_k) = \langle x^2, y^2, [y,x] \rangle \; P_3(G_k) %
    & \text{with} \quad %
    & P_2(G_k)/P_3(G_k)  \cong C_2 \times C_2 \times C_2,
 \end{align*}
 and, for $3 \le i \le 2^k+1$, the $i$th term is
 $P_i(G_k) = \langle x^{2^{i-1}} \rangle \gamma_i(G_k)$ so that
 \[
 P_i(G_k) =
 \begin{cases}
   \langle x^{2^{i-1}}, [y,x,\overset{i-1}{\ldots},x] \rangle
   \; P_{i+1}(G_k), & \text{if $i \equiv_2 0$ and $i \le k+1$,} \\
   \langle x^{2^{i-1}}, [y,x,\overset{i-1}{\ldots},x],
   [y,x,\overset{i-2}{\ldots},x,y] \rangle \; P_{i+1}(G_k) & \text{if
     $i \equiv_2 1$ and $i \le k+1$,} \\
   \langle [y,x,\overset{i-1}{\ldots},x] \rangle
   \; P_{i+1}(G_k), & \text{if $i \equiv_2 0$ and $k+1<i\le 2^k$,} \\
   \langle [y,x,\overset{i-1}{\ldots},x],
   [y,x,\overset{i-2}{\ldots},x,y] \rangle \; P_{i+1}(G_k) & \text{if
     $i \equiv_2 1$ and $k+1<i\le 2^k$} \\
    \langle [y,x,\overset{i-1}{\ldots},x] \rangle \; P_{i+1}(G_k) & \text{if
     $i=2^k+1$} 
 \end{cases}
 \]
 with
 \[
 P_i(G_k)/P_{i+1}(G_k) \cong 
 \begin{cases}
   C_2 \times C_2 & \text{if $i \equiv_2 0$ and $i \le k+1$,} \\
   C_2 \times C_2 \times C_2 & \text{if $i \equiv_2 1$ and
     $i \le k+1$,} \\
   C_2 & \text{if $i \equiv_2 0$ and $k+1<i\le 2^k$,} \\
   C_2 \times C_2 & \text{if $i \equiv_2 1$ and $k+1<i\le 2^k$} \\
    C_2  & \text{if $i=2^k+1$}.
 \end{cases}
 \]
\end{corollary}

\begin{proof}
  The descriptions of $G_k / P_2(G_k)$ and $P_2(G_k) / P_3(G_k)$ are
  straightforward.  Let $i \ge 3$.  Clearly,
  $P_i(G_k) \supseteq \langle x^{2^{i-1}} \rangle \gamma_i(G_k)$.  In
  view of Proposition~\ref{pro:lower-central-Gk}, it suffices to prove
  that $x^{2^{i-1}}$ is central modulo~$\gamma_{i+1}(G_k)$.  Indeed,
  from 
  Proposition~\ref{pro:standard-commutator-id}  we
  obtain
  \[
  [x^{2^{i-1}},y] \equiv [x,y]^{2^{i-1}} [x,y,x]^{\binom{2^{i-1}}{2}} \,
               \cdots \, [x, y, x, \overset{i-2} \ldots, x]^{\binom{2^{i-1}}{i-1}}
                              \quad \text{modulo }
  \gamma_{i+1}(G_k). 
  \]
  We have $4 \mid \binom{2^{i-1}}{r}$ for $2\le r \le i-1$ when $i\ge 4$, and the result follows since the exponent of $\gamma_2(G_k)$ is 4. For $i=3$, 
  \[
  [x^{4},y] \equiv [x,y]^{4} [x,y,x]^{\binom{4}{2}} \equiv 1
                              \quad \text{modulo }
  \gamma_{4}(G_k),
  \]
  as $\gamma_3(G_k)/\gamma_4(G_k)$ has exponent 2.
\end{proof}

\begin{corollary}\label{cor:dim-subgroup}
  The dimension subgroup series of $G_k$ has length~$2^k+2$.  For
  $1\le i \le 2^k+2$, the $i$th term is
  $D_i(G_k) =G_k^{\, 2^{l(i)}} \gamma_{\lceil \nicefrac{i}{2}\rceil}(G_k)^2\gamma_i(G_k)$, where
  $l(i) = \lceil\log_2 i \rceil$.

  Furthermore, if $i$ is not a power of $2$, equivalently if $l(i+1)=l(i)$, then $D_i(G_k)/D_{i+1}(G_k) \cong \gamma_{\lceil \nicefrac{i}{2}\rceil}(G_k)^2\gamma_i(G_k) /\gamma_{\lceil \nicefrac{i+1}{2}\rceil}(G_k)^2 \gamma_{i+1}(G_k)$  so that 
   \[
  D_i(G_k) =
  \begin{cases}
      \langle [y,x,\overset{i-1}{\ldots},x] \rangle D_{i+1}(G_k) & \text{if $i\equiv_2 1$,} \\
      \langle [y,x,\overset{i-3}{\ldots},x,y], [y,x,\overset{i-1}\ldots,x] \rangle  D_{i+1}(G_k) &
    \text{if $i\equiv_2 0$,}
  \end{cases}
  \]
  with
  \[
  D_i(G_k)/D_{i+1}(G_k) \cong
 \begin{cases}
                              C_2 &
                              \text{if $i\equiv_2 1$   and $i<2^k$,}\\
                            C_2\times C_2 &
                              \text{if $i\equiv_2 0$  and $i<2^k$,}\\
                               1 &
                              \text{if $i=2^k+1$,}\\
                            C_2 &
                              \text{if $i=2^k+2$,}\\
  \end{cases}
  \]
 whereas if $i = 2^l$ is a power of~$2$, equivalently if
  $ l(i+1) = l + 1$ for $l = l(i)$, then 
  $D_i(G_k)/D_{i+1}(G_k) \cong \langle x^{2^l} \rangle / \langle
  x^{2^{l+1}} \rangle \times \langle y^{2^l} \rangle / \langle
  y^{2^{l+1}} \rangle \times \langle [y,x,\overset{i-3}\ldots,x,y]\rangle\gamma_i(G_k) / \gamma_{i+1}(G_k)$ so that
  \begin{align*}
    & D_1(G_k) = \langle x,y \rangle D_2(G_k), \\
    & D_2(G_k) = \langle
      x^2,y^2, [y,x] \rangle D_{3}(G_k), \\
    & D_i(G_k) = \langle x^{2^l}, [y,x,\overset{i-3}\ldots,x,y], [y,x,\overset{i-1}{\ldots},x] \rangle D_{i+1}(G_k)
  \end{align*}
  with
  \[
  D_i(G_k)/D_{i+1}(G_k)\cong
  \begin{cases}
    C_2 \times C_2  & \text{if $i=1$, equivalently if $l = 0$,} \\
    C_2 \times C_2 \times  C_2  & \text{if $i=2$,
      equivalently if $l = 1$,} \\
    C_2 \times C_2  \times C_2 & \text{if $i = 2^l$ with $2 \le l \le k$.}
  \end{cases}
  \]
  
  In particular, for $2^{k-1}+1 \le i \le 2^k$ and thus $l(i)=k$,
  \[
  D_i(G_k) = G_k^{\, 2^k} \gamma_i(G_k) = \langle x^{2^k}, [y,x, \overset{2^k-3}{\ldots}, x,y] \rangle \;
  \gamma_i(G_k),
  \]
  so that 
  \[
  \log_2\lvert D_i(G_k) \rvert =\begin{cases}
        \log_2\lvert \gamma_i(G_k) \rvert +1  & \text{if $i\ne 2^k$,} \\
     \log_2\lvert \gamma_i(G_k) \rvert +2 & \text{if $i = 2^k$.}
  \end{cases}
    \]
\end{corollary}

\begin{proof}
  For $i \in \mathbb{N}$ write $l(i) = \lceil\log_2 i \rceil$. From
  \cite[Thm.~11.2]{DDMS99}, Lemma~\ref{lem:squared-commutators} and the fact that $\gamma_2(G_k)$ has exponent 4, we obtain
  \begin{align*}
  D_i(G_k) &= G_k^{\, 2^{l(i)}} \gamma_{\lceil \nicefrac{i}{2}\rceil}(G_k)^2\gamma_i(G_k)\\
  &=\begin{cases}
   G_k^{\, 2^{l(i)}} \gamma_i(G_k)  & \text{if $i\equiv_2 1$,} \\
   G_k^{\, 2^{l(i)}}\langle [y,x,\overset{i-3}\ldots,x,y]\rangle\gamma_i(G_k)  & \text{if $i\equiv_2 0$.} 
  \end{cases}
  \end{align*}
  In particular,
  $D_i(G_k) = 1$ for $i > 2^k+3$, by
  Proposition~\ref{pro:lower-central-Gk} and
  Corollary~\ref{cor:lower-p-central-Gk}.

  Now suppose that $1\le i \le 2^k+2$ and put $l = l(i)$.  As in
  Lemma~\ref{lem:contained} we observe that
  $G_k^{\, 2^l} \cap \gamma_i(G_k) \subseteq \langle [y,x,\overset{2^l-3}\ldots,x,y]\rangle \gamma_{2^l}(G_k)$.  If
  $l(i+1) = l$, then  
  \begin{align*}
   & D_i(G_k)/D_{i+1}(G_k) \\
   & = \begin{cases}
    									G_k^{\, 2^l}\gamma_i(G_k)/G_k^{\, 2^l}  \langle [y,x,\overset{i-2}\ldots,x,y]\rangle\gamma_{i+1}(G_k) & \text{if $i\equiv_2 1$,}\\
    									G_k^{\, 2^l}\langle [y,x,\overset{i-3}\ldots,x,y]\rangle\gamma_i(G_k)/G_k^{\, 2^l}  \gamma_{i+1}(G_k) & \text{if $i\equiv_2 0$,}
    									\end{cases}\\
                          & \cong
                          \begin{cases}
                            \gamma_i(G_k) / (G_k^{\, 2^l} \cap
                            \gamma_i(G_k)) \langle [y,x,\overset{i-2}\ldots,x,y]\rangle\gamma_{i+1}(G_k)& \text{if $i\equiv_2 1$,}\\
                           \langle [y,x,\overset{i-3}\ldots,x,y]\rangle \gamma_i(G_k) / (G_k^{\, 2^l} \cap
                            \langle [y,x,\overset{i-3}\ldots,x,y]\rangle\gamma_i(G_k)) \gamma_{i+1}(G_k)& \text{if $i\equiv_2 0$,}
                            \end{cases}\\
                              &\cong 
                              \begin{cases}
                              \gamma_i(G_l) / \langle [y,x,\overset{i-2}\ldots,x,y]\rangle\gamma_{i+1}(G_l) &
                              \text{if $i\equiv_2 1$,}\\
                             \langle [y,x,\overset{i-3}\ldots,x,y]\rangle \gamma_i(G_l) / \gamma_{i+1}(G_l) &
                              \text{if $i\equiv_2 0$,}
                              \end{cases}\\
                              &\cong 
                              \begin{cases}
                              C_2 &
                              \text{if $i\equiv_2 1$   and $i<2^k$,}\\
                            C_2\times C_2 &
                              \text{if $i\equiv_2 0$  and $i<2^k$,}\\
                               1 &
                              \text{if $i=2^k+1$,}\\
                            C_2 &
                              \text{if $i=2^k+2$.}\\
                              \end{cases}
  \end{align*}

  Now suppose that $l(i+1) = l+1$, equivalently $i = 2^l$.  We observe
  that, modulo $H_k$, the $i$th factor of the dimension subgroup
 series is
  \[
  D_i(G_k) H_k / D_{i+1}(G_k) H_k = \langle x^{2^l} \rangle H_k /
  \langle x^{2^{l+1}} \rangle H_k \cong C_2.
  \]
  Comparing with the overall order of~$G_k$, conveniently implicit in
  Corollary~\ref{cor:lower-p-central-Gk}, we deduce that
  \begin{align*} 
    D_i(G_k)/D_{i+1}(G_k) & = G_k^{\, 2^l}\gamma_i(G_k) / G_k^{\, 2^{l+1}}
                            \gamma_{i+1}(G_k) \\
                          & = \langle x^{2^l}, y^{2^l}, [y,x,\overset{i-3}\ldots,x,y]\rangle \gamma_i(G_l) / \langle
                            x^{2^{l+1}}, y^{2^{l+1}} \rangle \gamma_{i+1}(G_l) \\
                          & \cong \langle x^{2^l} \rangle / \langle
                            x^{2^{l+1}} \rangle \times \langle y^{2^l} \rangle / \langle
                            y^{2^{l+1}} \rangle \times \langle [y,x,\overset{i-3}\ldots,x,y]\rangle\gamma_i(G_l) / \gamma_{i+1}(G_l).
  \end{align*}
  All remaining assertions follow readily from
  Proposition~\ref{pro:lower-central-Gk}.
  \end{proof}

\begin{proposition}\label{pro:Frattini-Gk}
  The Frattini series of $G_k$ has length $k+1$ and satisfies:
  \begin{align*}
    & G_k=\Phi_0(G_k)=\langle x,y\rangle \Phi_1(G_k) \quad
      \text{with}\quad G_k/\Phi_1(G_k) \cong C_2 \times C_2, \\
    & \Phi_1(G_k) = \langle x^2,y^2, [y,x], [y,x,x] \rangle\Phi_2(G_k) \\
    & \qquad \text{with}\quad  \Phi_1(G_k)/\Phi_2(G_k) \cong C_2^{\, 4}, 
  \end{align*}
  and, for $2\le i< k$, the $i$th term is
  \begin{align*}
    & \Phi_i(G_k)=\langle x^{2^i},[y,x,
      \overset{\nu(i)}{\ldots}, x], [y,x,
      \overset{\nu(i)+1}{\ldots}, x], \ldots, [y,x,
      \overset{\nu(i+1)-1}{\ldots},
      x], \\
    & \qquad  [y,x, \overset{2\nu(i-1)-1}{\ldots}, x,y], [y,x,
      \overset{2\nu(i-1)+1}{\ldots}, x,y],
      \ldots, [y,x, \overset{2\nu(i)-3}{\ldots},
      x,y]\rangle \Phi_{i+1}(G_k) \\
    & \quad \text{with} \quad \Phi_i(G_k)/\Phi_{i+1}(G_k) \cong
              C_2^{\, 2^i+2^{i-1}+1} 
             \end{align*}
  where $\nu(j) = 2^j-1$ for $1 \le j \le k$, and
  \begin{align*}
  & \Phi_{k}(G_k)=\langle x^{2^k},[y,x,
      \overset{2^k-1}{\ldots}, x], [y,x,
      \overset{2^k}{\ldots}, x],   [y,x, \overset{2^k-3}{\ldots}, x,y]\rangle \Phi_{k+1}(G_k)\\
 & \quad \text{with} \quad \Phi_k(G_k)/\Phi_{k+1}(G_k) \cong
              C_2^{\,4}.
      \end{align*}    
     \end{proposition}

\begin{proof}
  As in the proof of Lemma~\ref{lem:squared-commutators}, we write $c_1=y$ and, for $i \ge 2$,
  \[
  c_i = [y,x,\overset{i-1}{\ldots},x] \qquad \text{and} \qquad z_i =
  [c_{i-1},y] = [y,x,\overset{i-2}{\ldots},x,y]
  \]
  and we have that
  \begin{equation} \label{equ:comm-ci-cj}
    [c_i,c_j] \equiv z_{i+j} \;\mathrm{mod}\;
      \gamma_{i+j+1}(G_k) \qquad \text{for $i > j \ge 1$.}  
  \end{equation}

  We use the generators specified in the statement of the proposition
  to define an ascending chain
  $1 = L_{k+1} \le \ldots \le L_1 \le L_0 = G_k$ so that
  each $L_i$ is the desired candidate for $\Phi_i(G_k)$.  For
  $1 \le i \le k$ we deduce from
  Proposition~\ref{pro:lower-central-Gk} and
  Corollary~\ref{cor:gamma-central-part} that
  \[
  L_i = \langle x^{2^i} \rangle M_i \quad \text{with} \quad M_i =
  \langle c_{\nu(i)+1} \rangle \gamma_{\nu(i)+2}(G_k) C_i
  \trianglelefteq G_k,
  \]
  where
  $C_i = \langle y^{2^i} \rangle \times \langle z_j \mid 2\nu(i-1)+1
  \le j \le 2^k \text{ and } j \equiv_2 1 \rangle$
  is central in~$G_k$.  (Note that the factor
  $\langle y^{2^i} \rangle$ vanishes if $i \ge 2$.)  Applying
  Proposition~\ref{pro:standard-commutator-id} and Lemma~\ref{lem:squared-commutators}, we see that
  $[x^{2^i},G_k] = [x^{2^i},H_k] \subseteq \gamma_{2^i+1}(G_k)$, hence
  $L_i \trianglelefteq G_k$ for $1 \le i \le k$.  Using
  also~\eqref{equ:comm-ci-cj}, we see that the factor groups
  $L_i/L_{i+1}$ are elementary abelian for $0 \le i \le k$.  In
  particular, this shows that $\Phi_i(G_k) \subseteq L_i$ for
  $1 \le i \le k+1$.

  Clearly, for each $i \in \{0,\ldots,k+1\}$, the value of
  $\log_2 \lvert L_i / L_{i+1} \rvert = d(L_i/L_{i+1})$ is bounded by
  the number of explicit generators used to define $L_i$ modulo
  $L_{i+1}$; these numbers are specified in the statement of the
  proposition and a routine summation shows that they add up to the
  logarithmic order $\log_2 \lvert G_k \rvert$, as given in
  Lemma~\ref{lem:order-Gk}.  Therefore each $L_i/L_{i+1}$ has the
  expected rank and it suffices to show that
  $\Phi_i(G_k) \supseteq L_i$ for $1 \le i \le k$.

  Let $i \in \{1,\ldots,k\}$.  It is enough to show that the
  following elements which generate $L_i$ as a normal subgroup belong
  to~$\Phi_i(G_k)$:
  \[
  x^{2^i}, \quad c_{\nu(i)+1}, \qquad \text{and} \quad z_j \quad
  \text{for} \quad \text{$2\nu(i-1)+1 \le j \le 2^k$ with
    $j \equiv_2 1$.}
  \]
  Clearly, $x^{2^i} \in \Phi_i(G_k)$. Also $(xy)^{2^i}, [y,x,x^2,\ldots,x^{2^{i-1}}] \in \Phi_i(G_k)$. Hence, applying the rule $(ab)^2=a^2b^2[b,a]$ repeatedly, we see that
  \[
  (xy)^{2^i}=x^{2^i}[y,x,x^2,\ldots,x^{2^{i-2}}]^2[y,x,x^2,\ldots,x^{2^{i-1}}]  \qquad \text{ for } i\ge 2,
  \]
   and by Proposition~\ref{pro:standard-commutator-id} and 
  Lemma~\ref{lem:squared-commutators}
  \[
  [y,x,x^2,\ldots,x^{2^{i-2}}]^2 \equiv z_{2^i-1}\quad \text{mod } \gamma_{2^i}(G_k).
  \]
   Now let $2\nu(i-1)+1 < j \le 2^k$ with $j \equiv_2 1$.  By
  Corollary~\ref{cor:gamma-central-part} and reverse induction on
  $j$ it suffices to show that $z_j$ is contained in $\Phi_i(G_k)$
  modulo $\gamma_{j+1}(G_k)$.  But this follows from
  \eqref{equ:comm-ci-cj} and the fact that $c_{\nu(i-1)+1},
  c_{j-\nu(i-1)-1} \in \Phi_{i-1}(G_k)$ by induction on~$i$.
  
   Lastly, similarly using Lemma~\ref{lem:squared-commutators}, 
  \[
  c_{\nu(i)+1} = [y,x,\overset{\nu(i)}{\ldots},x] \equiv
  [y,x,x^2,\ldots,x^{2^{i-1}}] \quad \text{mod } C_i.
  \]
  
  For $i=k$, recall that $[y,x,\overset{2^k}\ldots,x]=[y,x,\overset{2^k-1}\ldots,x,y]$. Thus we are done.
  
  We end by noting that  $\Phi_i(G_k)=\langle x^{2^i},[y,x, \overset{2^i-3}{\ldots}, x,y]\rangle \gamma_{2^i}(G_k)$ for $2\le i\le k$.
  \end{proof}

Using \cite[Cor.~4.2]{KT}, we can now complete
the proof of Theorem~\ref{thm:main-thm}: it suffices to compute
$\hdim^{\mathcal{S}}_G(Z)$ and $\hdim^{\mathcal{S}}_G(H)$ for the
standard filtration series
$\mathcal{S} \in \{ \mathcal{L}, \mathcal{D}, \mathcal{F} \}$.

Corollary~\ref{cor:lower-p-central-Gk} implies
\begin{align}
  \hdim^{\mathcal{L}}_G(Z) & = \varliminf_{i \to \infty}
                             \frac{\log_2 \lvert Z P_i(G) : P_i(G) \rvert}{\log_2 \lvert G : P_i(G)
                             \rvert} = \lim_{i \to \infty}
                             \frac{\nicefrac{i}{2}}{\nicefrac{5i}{2}} = \nicefrac{1}{5}, 
                             \label{equ:xi-for-L} \\
  \hdim^{\mathcal{L}}_G(H) & = \varliminf_{i \to \infty}
                             \frac{\log_2 \lvert H P_i(G) : P_i(G) \rvert}{\log_2 \lvert G : P_i(G)
                             \rvert} = \lim_{i \to \infty}
                             \frac{\nicefrac{3i}{2}}{\nicefrac{5i}{2}} = \nicefrac{3}{5}.\notag
\end{align}

Corollary~\ref{cor:dim-subgroup} implies
\begin{align}
  \hdim^{\mathcal{D}}_G(Z) & = \varliminf_{i \to \infty}
                             \frac{\log_2 \lvert Z D_i(G) : D_i(G) \rvert}{\log_2 \lvert G : D_i(G)
                             \rvert} = \lim_{i \to \infty}
                             \frac{\nicefrac{i}{2}}{\nicefrac{3i}{2}} = \nicefrac{1}{3},\label{equ:xi-for-D} \\
  \hdim^{\mathcal{D}}_G(H) & = \varliminf_{i \to \infty}
                             \frac{\log_2 \lvert H D_i(G) : D_i(G) \rvert}{\log_2 \lvert G : D_i(G)
                             \rvert} = \lim_{i \to \infty}
                             \frac{\nicefrac{3i}{2}}{\nicefrac{3i}{2}} = 1.\notag
\end{align}

Lastly, Proposition~\ref{pro:Frattini-Gk} implies
\begin{align}
  \hdim^{\mathcal{F}}_G(Z) & = \varliminf_{i \to \infty}
                             \frac{\log_2 \lvert Z \Phi_i(G) :
                             \Phi_i(G) \rvert}{\log_2 \lvert G :
                             \Phi_i(G) \rvert} = \lim_{i \to \infty}
                             \frac{2^{i-1}}{3\cdot 2^{i-1}} = \nicefrac{1}{3},\label{equ:xi-for-F} \\
  \hdim^{\mathcal{F}}_G(H) & = \varliminf_{i \to \infty}
                             \frac{\log_2 \lvert H \Phi_i(G) :
                             \Phi_i(G) \rvert}{\log_2 \lvert G :
                             \Phi_i(G) \rvert} = \lim_{i \to \infty}
                             \frac{3\cdot 2^{i-1}}{3\cdot 2^{i-1}} = 1.\notag
\end{align}


\section{The entire Hausdorff spectrum of $G$ with respect to the
  standard filtration series} \label{sec:entire-spectrum}

We continue to use the notation set up in
Section~\ref{sec:presentation} to study and determine the entire
Hausdorff spectrum of the pro-$2$ group~$G$, with respect to the
standard filtration series
$\mathcal{P}, \mathcal{D}, \mathcal{F},
\mathcal{L}$.

\begin{proof}[Proof of Theorem~\ref{thm:entire-spectrum}]
  As in Section~\ref{sec:presentation}, we write
  $W = G/Z \cong C_2 \mathrel{\hat{\wr}} \mathbb{Z}_2$, and we denote
  by $\pi \colon G \rightarrow W$ the canonical projection with
  $\ker \pi = Z$.

  \medskip

  If $\mathcal{S}$ is one of the filtration series
  $\mathcal{P}, \mathcal{D}, \mathcal{F}$ on~$G$, then the result follows exactly as was done in \cite{KT}.

  \medskip

  It remains to pin down the Hausdorff spectrum of $G$ with respect to
  the lower $2$-series $\mathcal{L} \colon P_i(G)$,
  $i \in \mathbb{N}$, on~$G$.  Now the
  normal subgroups $Z, H \trianglelefteq_\mathrm{c} G$ have strong
  Hausdorff dimensions $\hdim^{\mathcal{L}}_G(Z) = \nicefrac{1}{5}$
  and $\hdim^{\mathcal{L}}_G(H) = \nicefrac{3}{5}$.  As in \cite{KT},
  we deduce that $\hspec^{\mathcal{L}}(G)$ contains
  \[
  S = [0,\nicefrac{3}{5}] \cup
  \{\nicefrac{3}{5} + \nicefrac{m}{5\cdot2^{n-1}}\mid m, n \in\mathbb{N}_0
  \text{ with } 2^{n-1} < m\le 2^n\}.
  \]
  Thus it suffices to show that
  \begin{equation} \label{equ:3/5-4/5} 
    (\nicefrac{3}{5},\nicefrac{4}{5}) \subseteq \hspec^{\mathcal{L}}(G)
    \subseteq  (\nicefrac{3}{5},\nicefrac{4}{5}) \cup S.
   \end{equation}
   
   The second inclusion follows as in \cite{KT}.
 To prove the first inclusion in~\eqref{equ:3/5-4/5}, we mimic the argument in \cite{KT}. until the key step of showing that
  \begin{equation}\label{equ:the-key}
    \varliminf_{i \to \infty} \frac{\log_2 \lvert K P_i(G) \cap Z : P_i(G) \cap Z
      \rvert}{\log_2 \lvert Z : P_i(G) \cap Z \rvert} =
    \hdim_Z^{\mathcal{L} \vert_Z}(K \cap Z) = (2m-1)/2^{n}.
  \end{equation}
   First we examine the lower limit on the left-hand side, restricting
  to indices of the form $i = 2^k +2$, $k \in \mathbb{N}$.  Let
  $i = 2^k+2$, where $k \ge n$.  Recall that
    $G_k = G / \langle x^{2^{k+1}}, [x^{2^k},y] \rangle^G$ 
  and consider the canonical projection $\varrho_k \colon G \to G_k$,
  $g \mapsto \overline{g}$.  As before, we write $H_k = H \varrho_k$. 
    Furthermore, we observe that
    $Z_k = \langle \overline{x}^{2^k} \rangle Z \varrho_k$ with
    $\lvert Z_k : Z \varrho_k \rvert = 2$. 
  By Corollary~\ref{cor:lower-p-central-Gk}, we have
  \[
  \lvert H_k : H_k \cap \underbrace{P_i(G_k)}_{=1} \rvert = \lvert H_k
  \rvert = \lvert   H : H \cap P_i(G) \rvert
  \]
  and hence
  \[
  \frac{\log_2 \lvert K P_i(G) \cap Z : P_i(G) \cap Z \rvert}{\log_2
    \lvert Z : P_i(G) \cap Z \rvert} = \frac{\log_2 \lvert K \varrho_k
    \cap  Z \varrho_k \rvert}{\log_2 \lvert
     Z \varrho_k \rvert}.
  \]
  Observe that
  \[
  K \varrho_k \cap H_k = \langle \overline{y_j} \mid 0 \le j < 2^k \text{ with } j
  \equiv_{2^n} 0,1, \ldots, m-1 \rangle.
  \]
  From Lemma~\ref{lem:order-Gk} we 
    see that $Z \varrho_k \cong C_2^{\, 2^{k-1}+1}$ 
  and further we deduce that
  \begin{align*}
    K & \varrho_k \cap Z \varrho_k \\
      & =  \langle \{ \overline{y}^2 \} 
        \cup \{ [ \overline{y_0}, \overline{ y_j}] \mid 0
        \le j \le 2^{k-1}, \; j \equiv_{2^n}   0, \pm 1,\ldots, \pm (m-1)   \}  
        \rangle \\
      & \cong  C_2^{\, (2m-1) 2^{k-n-1} + 1}.
  \end{align*}
  This yields
  \begin{align*}
    \varliminf_{i \to \infty} \frac{\log_2 \lvert K P_i(G) \cap Z :
    P_i(G) \cap Z \rvert}{\log_2 \lvert Z : P_i(G) \cap Z \rvert}
    & \le \varliminf_{k \to \infty} \frac{\log_2 \lvert K \varrho_k \cap
     Z \varrho_k \rvert}{\log_2 \lvert
       Z \varrho_k \rvert}\\
    & = \lim_{k \to \infty}  \frac{(2m-1) 2^{k-n-1} + 1}{2^{k-1}+1} = (2m-1)/2^n.
  \end{align*}
  In order to establish~\eqref{equ:the-key} it now suffices to prove
  that
  \begin{equation} \label{equ:last-reduction}
  \varliminf_{i \to \infty} \frac{\log_2 \lvert (K \cap Z) (P_i(G)
    \cap Z) : P_i(G) \cap Z \rvert}{\log_2 \lvert Z : P_i(G) \cap Z
    \rvert} \ge (2m-1)/2^n.
  \end{equation}
  Our analysis above yields
  \[
  K \cap Z = \langle \{ y^2\} \cup \{ [y_0, y_j] \mid j \in \mathbb{N}
  \text{ with } j \equiv_{2^n} 
  0, \pm 1,\ldots,\pm (m-1) \} \rangle.
  \]
  Setting
  \[
  L= \langle y_j \mid j \in \mathbb{N}_0 \text{ with } j \equiv_{2^n} 
  0, \pm 1, \ldots, \pm (m-1) \rangle Z,
  \]
  and recalling the notation $c_1 = y = y_0$, we conclude that
  \[
  K \cap Z \supseteq \{ [c_1,g] \mid g \in L \}.
  \] 

  Next we consider the set
  \[
  D = \{ j \in \mathbb{N}_0 \mid \exists g \in L :  g \equiv_{P_{j+1}(G) Z} c_j \}.
  \]
  Each element $y_j$ can be written (modulo $Z$) as a product
  \[
  y_j \equiv_Z \prod_{k=0}^j c_{k+1}^{\, \beta(j,k)} \qquad
  \text{where $\beta(j,k) = \tbinom{j}{k}$,}
  \]
  using the elements $c_i = [y,x,\overset{i-1}{\ldots},x]$ introduced
  in the proof of Proposition~\ref{pro:Frattini-Gk}.  In this product
  decomposition, the exponents should be read modulo~$2$, and the
  elementary identity $(1+t)^{j + 2^n} = (1+t)^j (1+t^{2^n})$ in
  $\mathbb{F}_2[\![t]\!]$ translates to
  \[
  y_j^{\, -1} y_{j + 2^n} = y^{-x^j} y^{x^{j + 2^n}} \equiv_Z \prod_{k=0}^j c_{k+1+2^n}^{\,
    \beta(j,k)}  \qquad \text{for all $j \in \mathbb{N}$.}
  \]
  Inductively, we obtain
  \[
  D = D_0 + 2^n \mathbb{N}_0 \qquad \text{for $D_0 = D \cap
    \{1,\ldots,2^n\}$}.
  \]
  One checks that $D_0 =\{1,2,\ldots, 2m-1\}$. Hence for each
  $k \in \mathbb{N}_0$, the set
  $((2k) 2^n + D_0) \cup ((2k+1)2^n + D_0)$ consists of 
  $2m-2$ even numbers.

  For each $j \in D$ with $j \equiv_2 0$ there exists $g_j \in L$ with
  $g_j \equiv_{P_{j+1}(G) Z} c_j$ and we deduce from~\eqref{equ:comm-ci-cj} that
  \[
  z_{1+j} \equiv_{P_{2+j}(G)} [c_1,c_j] \equiv_{P_{2+j}(G)}
  [c_1,g_j] \in K \cap Z.
  \]
  For $i = 2^{n+1}q + r \in \mathbb{N}$, where $q \in \mathbb{N}$, $r \in \mathbb{N}_0$
  with $0 \le r < 2^{n+1}$, the count
  \[
  \lvert \{ j \in D \mid j \equiv_2 0 \text{ and } j < i-1 \} \rvert
  \geq q (2m-2) -1,
  \]
 yields
  \[
  \log_2 \lvert (K \cap Z) (P_i(G) \cap Z) : P_i(G) \cap Z \rvert
  \ge q (2m-2) -1.
  \]
 However,  we deduce from the proof of Proposition~\ref{pro:lower-central-Gk} that for $j\equiv_{2^n} 2m-3 $, the element $z_{1+j}\equiv z_{4+j}$ modulo $P_{5+j}(G)$. Hence we in fact have
  \[
  \log_2 \lvert (K \cap Z) (P_i(G) \cap Z) : P_i(G) \cap Z \rvert
  \ge q (2m-1) -1.
  \]

  From Corollary~\ref{cor:lower-p-central-Gk} we observe that, for
  $i \ge 3$,
  \[
  \log_2 \lvert Z : P_i(G) \cap Z \rvert = \lfloor \nicefrac{i}{2} 
  \rfloor \le
  q\cdot 2^n + 2^n. 
  \]
  These estimates show that \eqref{equ:last-reduction} holds.
  \end{proof}


\end{document}